\documentclass[final,leqno,onefignum,onetabnum]{siamltex1213}
\usepackage{algorithm}
\usepackage{algorithmic}
\title{SIMULTANEOUS DIAGONALIZATION OF MATRICES AND ITS APPLICATIONS IN QUADRATICALLY CONSTRAINED QUADRATIC PROGRAMMING\thanks{This work was partially supported by Research Grants Council of Hong Kong under grant 414513. The second author is also grateful to the support from the Patrick Huen Wing Ming Chair Professorship of Systems Engineering \& Engineering Management.}}

\author{Rujun Jiang\thanks{Department of Systems Engineering and Engineering Management, The Chinese University of Hong Kong, Shatin, N. T., Hong Kong (\email{rjjiang@se.cuhk.edu.hk})} \and Duan Li\thanks{Corresponding author. Department of Systems Engineering and Engineering Management, The Chinese University of Hong Kong, Shatin, N. T., Hong Kong (\email{dli@se.cuhk.edu.hk})}}

\newtheorem{asmp}{Assumption}[section]

\newtheorem{rmk}{Remark}[section]

\def\Dg{{\rm Diag}}
\begin{document}

\maketitle
\slugger{mms}{xxxx}{xx}{x}{x--x}

\begin{abstract}
 An equivalence between attainability of simultaneous diagonalization (SD) and hidden convexity in quadratically constrained quadratic programming (QCQP) stimulates us to investigate necessary and sufficient SD conditions, which is one of the open problems posted by
Hiriart-Urruty [SIAM Rev., 49 (2007), pp. 255--273] nine years ago.  In this paper we give a necessary and sufficient SD condition for any two real symmetric matrices and offer a necessary and sufficient SD condition for any finite collection of real symmetric matrices under the existence assumption of a semi-definite matrix pencil. Moreover, we apply our SD conditions to QCQP, especially with one or two quadratic constraints, to verify the exactness of its second-order cone programming relaxation and to facilitate the solution process of QCQP.
\end{abstract}



\pagestyle{myheadings}
\thispagestyle{plain}
\markboth{SIMULTANEOUS DIAGONALIZATION}{RUJUN JIANG AND DUAN LI}

\section{Introduction}

Recently, Ben-Tal and Hertog \cite{B1} show that if the two matrices in
the objective function and the constraint of a quadratically constrained quadratic programming (QCQP) problem are simultaneously diagonalizable (SD), this QCQP problem can be
then turned into an equivalent second-order cone programming (SOCP) problem, which can be solved much faster than the semi-definite programming (SDP) reformulation.
(Note that to simply the notation, we use the abbreviation SD to denote both ``simultaneously diagonalizable'' and ``simultaneous diagonalization'' via congruence when no confusion will arise.)
A natural question to ask is when a given pair of matrices are SD. Essentially, the SD problem of a finite collection of symmetric matrices via congruence is one of the 14 open problems posted by
Hiriart-Urruty \cite{H} nine years ago: ``A collection of $m$ symmetric $(n,n)$ matrices $\{A_1, A_2, \ldots, A_m\}$ is said to be {\it simultaneously diagonalizable via congruence} if there exists a nonsingular matrix $P$ such that each of the $P^TA_iP$ is diagonal.''

Two sufficient conditions for SD of two matrices $A$ and $B$ have been developed in \cite{H, M, W, S}: i) There exist $\mu_1$ and $\mu_2$ $\in$ $\Re$ such that $\mu_1A+\mu_2B\succ0$, which is termed as {\it regular case} in \cite{M, W,S}; and ii) When $n\geq3$, $(\langle Ax,x\rangle=0 ~{\rm and}~ \langle Bx,x\rangle=0)$ implies $(x=0)$. Lancaster and Rodman \cite{LR} offer another sufficient condition for SD of two matrices $A$ and $B$: There exist $\alpha$ and $\beta$ such that $C: = \alpha A+\beta B\succeq0$ and $Ker(C)\subseteq Ker(A)\cap Ker(B)$. Several other SD sufficient conditions of two matrices are presented in Theorem 4.5.15 of \cite{HJ}. A necessary and sufficient condition for SD of two matrices has been derived by Uhlig in \cite{Uhlig2,Uhlig} if at least one of the two matrices is nonsingular.
A sufficient SD condition is well known for multiple matrices: If $m$ matrices commute with each other, then they are SD, see \cite{HJ} (Problems 22 and 23, Page 243).
  The equivalence between the attainability of SD and a hidden convexity in QCQP stimulates the research in this paper.

In this paper, we provide a necessary and sufficient SD condition for any two real symmetric matrices, which extends the result in \cite{Uhlig} to any two arbitrary matrices. We then show its applications in QCQP with one quadratic inequality constraint or with an interval constraint. Furthermore, we find a necessary and sufficient SD condition for $m$ ($m\geq3$) real symmetric matrices when there is a semi-definite matrix pencil, i.e., there exist $\lambda_1,\ldots,\lambda_m\in \Re$ (not all of which are zero), such that $\lambda_1A_1+\lambda_2A_2+\cdots+\lambda_mA_m\succeq0$. The semi-definite matrix pencil has been a widely adopted assumption in the solvability study in SDP and its duality \cite{KS,St}. We also show its applications in QCQP and show that QCQP with two quadratic constraints has a non-trivial exact relaxation. Our proofs are constructive for both cases of two matrices and multiple (more than two) matrices. Thus, we essentially establish systematic procedures for testing SD using the derived sufficient and necessary condition for any finite number of symmetric matrices.

 Our paper is organized as follows. In Section 2, we provide a necessary and sufficient SD condition for any pair of real symmetric matrices and show its applications in QCQP with one quadratic inequality  constraint or an interval quadratic constraint. In Section 3, we show necessary and sufficient SD conditions for a finite number of real symmetric matrices under the definite matrix pencil condition and a more general semi-definite matrix pencil condition. We next discuss the applications of SD in QCQP with multiple, in particularly 2, quadratic constraints in Section 4. Finally, we conclude our paper in Section 5.

\textbf{Notation.} We use $I_p$ and $0_p$ to denote the identity matrix of dimension $p$ $\times$ $p$ and the zero square matrix of dimension $p$ $\times$ $p$, respectively, and use $0_{p \times q}$ to denote a zero matrix of dimension $p$ $\times$ $q$. The notation $\Re^n$ represents
the $n$ dimensional vector space and $S^n$ represents the $n\times n$ symmetric matrix space.
We denote by $tr(A)$ the trace of a square matrix $A$,
and  $\diag(A_1,A_2,\ldots,A_k)$ the block diagonal matrix
$$\left(\begin{array}{cccc}
A_1&&&\\
&A_2&&\\
&&\ddots&\\
&&&A_k
\end{array}\right),$$
where $A_i$, $i=1,\ldots,k$, are all square matrices.
 Finally, we use $\Dg(X)$ to denote the vector formed by the diagonal elements of the square matrix $X$ with $\Dg(X)_k$ representing its $k$th entry.

\section{Simultaneous diagonalization for two matrices}
In this section, we first derive a necessary and sufficient simultaneous condition for any two matrices and then show its applications in the trust region subproblem (TRS) and its variants.
\subsection{SD condition}
The following lemma from \cite{Uhlig} shows a necessary and sufficient condition for SD of two matrices if at least one of the two matrices is nonsingular.
\begin{lemma}\label{Lemma1}\cite{Uhlig}
 If one of the two real symmetric matrices $A$ and $B$ is nonsingular (without loss of generality, we assume that $A$ is nonsingular), they are SD if and only if the Jordan normal form of $A^{-1}B$ is diagonal.
\end{lemma}

The following lemma is also useful in deriving the congruent matrix to achieve the SD for two matrices.
\begin{lemma}
\label{lem2} (Lemma 1 in \cite{Uhlig2})
Let a Jordan matrix be denoted by $$K:=\diag(C(\lambda_1),C(\lambda_2),\ldots,C(\lambda_k)),$$
where $C(\lambda_i):=\diag(K_{i_1}(\lambda_i),K_{i_2}(\lambda_i),\ldots,K_{i_t}(\lambda_i))$ denotes all the Jordan blocks associated with eigenvalue $\lambda_i$, and
$$K_{i_j}:=\left(\begin{array}{cccccc}
\lambda_i&1&&&\\
&\lambda_i&1&&\\
&&\ddots&\ddots&\\
&&&\lambda_i&1\\
&&&&\lambda_i
\end{array}\right)_{i_j\times i_j},$$
$j=1,2,\ldots,t.$

 If $SK$ is symmetric for a symmetric matrix $S$,  then $S$ is  block diagonal,  $S=\diag(S_1,S_2,\ldots,S_k),$
with $\dim S_i=\dim C(\lambda_i)$.
\end{lemma}

When the condition in Lemma \ref{Lemma1} is satisfied, we can identify a congruent matrix that makes two $n$ $\times$ $n$ symmetric matrices $A$ and $B$ SD as follows.

 When $A$ is nonsingular and the Jordan normal form, denoted by $J$, of $A^{-1}B$ is diagonal, there exists a nonsingular matrix $P$ such that
 $$J:=P^{-1}A^{-1}BP=\diag(\lambda_1I_{n_1},\ldots,\lambda_kI_{n_k}),$$ where $k$ $\leq$ $n$.
 Noting the facts that $P^TAP$ is symmetric and $(P^TAP)J = P^TBP$ is symmetric. Applying Lemma \ref{lem2}, we have the following two block diagonal matrices,
  $$P^TAP=\diag(A_1,\ldots,A_k) {\rm~ and~ } P^TBP=\diag(B_1,\ldots,B_k),$$
 where $\diag(A_1,\ldots,A_k)$ and $\diag(B_1,\ldots,B_k)$ have the same partition as $J$, i.e., $\dim A_i=\dim B_i=\dim I_{n_i}$, $i=1,\ldots,k$. Moreover, \begin{eqnarray*}
J&=&P^{-1}A^{-1}BP=(P^TAP)^{-1}P^TBP\\
&=&(\diag(A_1,\ldots,A_k))^{-1}\diag(B_1,\ldots,B_k)\\
&=&\diag(A_1^{-1}B_1,\ldots,A_k^{-1}B_k).
\end{eqnarray*}
Thus, for $i=1,\ldots,k$, $A_i^{-1}B_i=\lambda_iI_{n_i}$, i.e., $\lambda_i A_i=B_i$. Let $R=\diag(R_1,\ldots,R_k)$ such that $R_i^TA_iR_i$ is a spectral decomposition of $A_i$, and then $R_i^TB_iR_i=\lambda_iR_i^TA_iR_i$ is also diagonal, $i=1,\ldots,k$.
So with the two nonsingular matrices $P$ and $R$, $(PR)^TAPR$ and $(PR)^TBPR$ are both diagonal, i.e., $A$ and $B$ are SD via the congruent matrix $PR$.

\begin{lemma}\label{lem2.3}
Let $A:=\diag(A_1,0_{q+r})$ {\rm~and~}$$B:=\left(\begin{array}{ccc}
B_1&0_{p \times q}&B_2\\
0_{q \times p}&B_3&0_{q \times r}\\
B_2^T&0_{r \times q}&0_r
\end{array}\right),$$ where both $A_1$ and $B_1$ are $p\times p$ nonsingular real symmetric matrices, $B_3$ is a $q\times q$ nonsingular real symmetric matrix, and $B_2$ is a $p\times r$ matrix with $r\leq p$ and with full column rank, then $A$ and $B$ cannot be SD.
\end{lemma}
\begin{proof}
 We will prove the lemma by contradiction. Let us assume that $A$ and $B$ are SD. Then, there exists a congruent matrix $$P:=\left(\begin{array}{ccc}
P_1&P_2&P_3\\
P_4&P_5&P_6\\
P_7&P_8&P_9\end{array}\right),$$ where the partition of $P$ is the same as the partition of $B$, such that $P^TAP$ and $P^TBP$ are diagonal, and without loss of generality, the nonzero elements of $P^TAP$ lie in the first $p$ entries of the diagonal
$$P^TAP= \left(\begin{array}{ccc}
P_1^TA_1P_1&P_1^TA_1P_2&P_1^TA_1P_3\\
P_2^TA_1P_1&P_2^TA_1P_2&P_2^TA_1P_3\\
P_3^TA_1P_1&P_3^TA_1P_2&P_3^TA_1P_3\end{array}\right).$$ Since $A_1$ and $P_1^TA_1P_1$ are nonsingular and $P_1$ is nonsingular, and thus $P_1^TA_1$ is nonsingular. As $P^TAP$ is diagonal, then $P_1^TA_1P_2$ and $P_1^TA_1P_3$ are both zero matrices. Thus, $P_2=0$ and $P_3=0$.

Furthermore, the following transformation holds for matrix $B$,
$$P^TBP=\left(\begin{array}{ccc}
\begin{array}{c}P_1^TB_1P_1+P_1^TB_2P_7\\+P_4^TB_3P_4+P_7^TB_2^TP_1\\\end{array}&P_1^TB_2P_8+P_4^TB_3P_5&P_1^TB_2P_9+P_4^TB_3P_6\\

P_8^TB_2^TP_1+P_5^TB_3^TP_4&P_5^TB_3P_5&P_5^TB_3P_6\\
P_9^TB_2^TP_1+P_6^T B_3^TP_4&P_6^TB_3^TP_5&P_6^TB_3P_6\end{array}\right).$$ From our assumptions, $B_1$ and $B_3$ are both nonsingular, and $B_2$ is of full column rank, which implies that $B$ is nonsingular. So the diagonal matrix $P^TBP$ must be nonsingular. As $P_5^TB_3P_5$ is nonsingular, $P_5$ is thus nonsingular. Note that $P_5^TB_3P_6=0$ implies that $P_6=0_{q\times r}$, which further leads to $P_6^TB_3P_6=0_r$. This contradicts the non-singularity of $P^TBP$.
\end{proof}

\begin{lemma}\label{lem2.4}
For any two $p\times p$ real symmetric matrices $A_1$ and $B_1$, and a $q\times q$ matrix $B_2$, if $A_1$ and $B_2$ are nonsingular and diagonal, and $B_1$ is nonsingular, then $A_1$ and $B_1$ are SD if and only if $A:=\diag(A_1,0_q)$ and $B:=\diag(B_1, B_2)$ are SD.
\end{lemma}
\begin{proof}
The ``$\Rightarrow$'' part: If $A_1$ and $B_1$ are SD via congruent matrix $S_{p \times p}$, then the congruent matrix $\diag(S_{p\times p},I_q)$ makes $A$ and $B$ SD.

The ``$\Leftarrow$'' part: If
$A$ and $B$
are SD, then there exists a nonsingular matrix $$P:=\left(\begin{array}{cc}
P_1&P_2\\
P_3&P_4
\end{array}\right),$$ such that, i) matrix
$$P^TAP=\left(\begin{array}{cc}
P_1^T&P_3^T\\
P_2^T&P_4^T
\end{array}\right)
\left(\begin{array}{cc}
A_1&\\
&0_q
\end{array}\right)
\left(\begin{array}{cc}
P_1&P_2\\
P_3&P_4
\end{array}\right)=\left(\begin{array}{cc}
P_1^TA_1P_1&P_1^TA_1P_2\\
P_2^TA_1P_1&P_2^TA_1P_2
\end{array}\right)$$ is diagonal, where, without loss of generality, the diagonal matrix $P_1^TA_1P_1$ is assumed to be nonsingular, and ii) matrix $P^TBP$ is diagonal.
Then, $P_1$ must be nonsingular and $P_2$ must be a zero matrix. We can now simplify $P^TBP$ to
 $$\left(\begin{array}{cc}
P_1^T&P_3^T\\
0_{q \times p}&P_4^T
\end{array}\right)
\left(\begin{array}{cc}
B_1&\\
&B_2
\end{array}\right)
\left(\begin{array}{cc}
P_1&0_{p \times q}\\
P_3&P_4
\end{array}\right)  =  \left(\begin{array}{cc}
P_1^T B_1P_1+P_3^TB_2P_3&P_3^TB_2P_4\\
P_4^TB_2P_3&P_4^TB_2P_4
\end{array}\right).$$ Since both $B_2$ and $P_4^TB_2P_4$ are nonsingular, $P_4$ is nonsingular. Thus, $P_3=0$, due to $P_3^TB_2P_4=0$. Finally, we conclude that $P_1^TB_1P_1$ is diagonal.
\end{proof}

\begin{lemma}\label{claim2-1}
For any two $n$ $\times$ $n$ singular real symmetric matrices $A$ and $B$,  there always exists a nonsingular matrix $U$ such that \begin{equation}\label{MA}
\tilde{A}:=U^TAU=\left(\begin{array}{cc}
A_1&0_{p \times (n-p)}\\
0_{(n-p) \times p} &0_{n-p}
\end{array}\right),
\end{equation} and
\begin{equation}\label{MB}
\tilde{B}:=U^TBU=\left(\begin{array}{ccc}
\mathcal{B}_1&0_{p \times q}&\mathcal{B}_2\\
0_{q \times p}&\mathcal{B}_3&0_{q \times r}\\
\mathcal{B}_2^T&0_{r \times q} &0_r
\end{array}\right),
\end{equation} where $p$, $q$, $r$ $\geq$ 0, $p + q +r=n$,  $A_1$ is a nonsingular diagonal matrix, $A_1$ and $\mathcal{B}_1$ have the same dimension of $p$ $\times$ $p$, $\mathcal{B}_2$ is a $p$ $\times$ $r$ matrix, and  $\mathcal{B}_3$ is a $q$ $\times$ $q$ nonsingular diagonal matrix.
\end{lemma}

\begin{proof}
We outline a proof for this lemma. Applying the spectral decomposition to $A$ identifies matrix $Q_1$ such that $\bar{A}:=Q_1^TAQ_1=\diag(A_1,0_{n-p})$,
where $A_1$ is a $p\times p$ nonsingular diagonal matrix. We express the corresponding $Q_1^TBQ_1$ as
$$\bar{B}:=Q_1^TBQ_1= \left(\begin{array}{cc}
B_1&B_2\\
B_2^T&B_3\\\end{array}\right),$$ where $B_1$ is of dimension $p\times p$.

If $B_3$ is not diagonal, we can further apply the spectral decomposition to $B_3$ by identifying a congruent matrix $Q_2:=\diag(I_p, V_{1})$ such that
$V_1^TB_3V_1$ is diagonal and the nonzero eigenvalues of $V_1^TB_3V_1$ are arranged to the upper left part of its diagonal.
Then $\hat{A}:=Q_2^T\bar{A}Q_2=\bar{A}$ and
$$\hat{B}:=Q_2^T\bar{B}Q_2=\left(\begin{array}{ccc}
B_1&B_4&B_5\\
B_4^T&B_6&0_{q \times (n-p-q)}\\
B_5^T&0_{(n-p-q) \times q}&0_{n-p-q}
\end{array}\right),$$ where $B_6$ is a nonsingular diagonal matrix of dimension $q$ $\times$ $q$.

Applying the congruent matrix
$$Q_3:=\left(\begin{array}{ccc}
I_p&0_{p\times q}&0_{p \times (n-p-q)}\\
-B_6^{-1}B_4^T&I_q&0_{q \times (n-p-q)}\\
0_{(n-p-q) \times p}&0_{(n-p-q) \times q}&I_{n-p-q}
\end{array}\right)$$ to both $\hat{A}$ and $\hat{B}$ yields $\tilde{A}:=Q_3^T\hat{A}Q_3=\hat{A}=\bar{A}$ and
$$\tilde{B}:=Q_3^T\hat{B}Q_3 =
\left(\begin{array}{ccc}
B_1-B_4B_6^{-1}B_4^T&0_{p \times q}&B_5\\
0_{q \times p}&B_6&0_{q \times (n-p-q)}\\
B_5^T&0_{(n-p-q) \times q}&0_{n-p-q}
\end{array}\right),$$ which take exactly the forms in (\ref {MA}) and (\ref{MB}), respectively.

In summary, letting $U := Q_1Q_2Q_3$ yields the forms of $U^TAU$ and $U^TBU$ given in (\ref{MA}) and (\ref{MB}) in the lemma, respectively.
\end{proof}

We present next a theorem which extends the results in \cite{Uhlig} to situations where both matrices are singular.
\begin{theorem}\label{Two-SD}
Two singular matrices $A$ and $B$, which take the forms (\ref {MA}) and (\ref{MB}), respectively, are SD if and only if the Jordan normal form of $A_1^{-1}\mathcal{B}_1$ is diagonal and $\mathcal{B}_2$ is a zero matrix or $r$ = 0 ($\mathcal{B}_2$ does not exist).
\end{theorem}
\begin{proof}
Recall that $A$ and $B$ are SD if and only if $\tilde{A}$ and $\tilde{B}$ are SD
(The notations $\tilde{A}$ and $\tilde{B}$  are the same as those in Lemma \ref{claim2-1}.).  We can always choose a sufficiently large
$\lambda$ such that the first $p$ columns of $\tilde{B}+\lambda \tilde{A}$ are linearly independent. For example, by setting
$\lambda=\max_{i=1,\ldots,p}\sum_{j=1}^p |b_{ij}|/|a_{ii}|+1 $, where $a_{ii}$ is the $i$th diagonal entry of the nonsingular diagonal matrix
$A_1$ and $b_{ij}$ is the $(i,j)$th entry of $\mathcal{B}_1$ and noting that $A_1$ is diagonal,
$\lambda A_1+\mathcal{B}_1$ becomes diagonally dominant and is thus nonsingular.

If $\mathcal{B}_2$ does not exist or is a zero matrix, $\tilde{A}$ and $\tilde{B}$ are SD $\Leftrightarrow$ $\tilde{B}+\lambda \tilde{A}$ and $\tilde{A}$ are SD ($\lambda>0$) $\Leftrightarrow$ $\lambda A_1+\mathcal{B}_1$ and $A_1$ are SD $\Leftrightarrow$ $\mathcal{B}_1$ and $A_1$ are SD $\Leftrightarrow$ the Jordan normal form of $A_1^{-1}\mathcal{B}_1$ is diagonal, where the second ``$\Leftrightarrow$'' follows Lemma \ref{lem2.4}.

If the columns of $\mathcal{B}_2$ are linearly independent, according to Lemma \ref{lem2.3}, $\tilde{A}$ and $\lambda \tilde{A}+\tilde{B}$ are not SD $\Rightarrow$ $\tilde{ A}$ and $\tilde{B}$ are not SD $\Rightarrow$ $A$ and $B$ are not SD.

If the columns of $\mathcal{B}_2$ are linearly dependent, we can find a nonsingular congruent matrix
$$Q_4 := \left(\begin{array}{ccc}
I_p&0_{p \times q}&0_{p \times (n-p-q)}\\
0_{q \times p}&I_q&0_{q \times (n-p-q)}\\
0_{(n-p-q) \times p} &0_{(n-p-q) \times q}&T_{(n-p-q)\times(n-p-q)}
\end{array}\right)$$
such that $\check{A}:=Q_4^T\tilde{A}Q_4=\tilde{A}$ and
$$\check{B} := Q_4^T\tilde{B}Q_4=\left(\begin{array}{cccc}
\mathcal{B}_1&0_{p \times q}&\mathcal{B}_4&0_{p \times (n-p-q-s)}\\
0_{q \times p}&\mathcal{B}_3&0_{q \times s}&0_{q \times (n-p-q-s)}\\
\mathcal{B}_4^T&0_{s \times q}&0_s&0_{s \times (n-p-q-s)}\\
0_{(n-p-q-s) \times p}&0_{(n-p-q-s) \times q}&0_{(n-p-q-s) \times s}&0_{n-p-q-s}
\end{array}\right),$$ where $T_{(n-p-q)\times(n-p-q) }$ is the product of elementary matrices which makes
$$\mathcal{B}_2T_{(n-p-q)\times(n-p-q) }=(\mathcal{B}_4,0),$$ and $\mathcal{B}_4$ is a $p$ $\times$ $s$ ($s$ $<$ $n-p-q$) matrix with full
column rank. Then according to Lemma \ref{lem2.3}, similar to the case where $\mathcal{B}_2$ is of full column rank, $A$ and
$B$ are not SD.

We complete the proof of the theorem. \end{proof}

Algorithm \ref{alg1} provides us an algorithmic procedure to verify whether two matrices $A$ and $B$ are SD.
\begin{algorithm}
\caption{Procedure to check whether two matrices $A$ and $B$ are SD}
\label{alg1}
\begin{algorithmic}[1]
\REQUIRE Two symmetric $n\times n$ matrices $A$ and $B.$
\STATE Apply the spectral decomposition to $A$
 such that $\bar{A}:=Q_1^TAQ_1=\diag(A_1,0)$, where $A_1$ is a nonsingular diagonal matrix, and express $\bar{B}:=Q_1^TBQ_1=\left(\begin{array}{cc}
B_1&B_2\\
B_2^T&B_3
\end{array}\right)$
\STATE Apply the spectral decomposition to $B_3$ such that $V_1^TB_3V_1=\left(\begin{array}{cc}
B_6&0\\
0&0
\end{array}\right)$ and $\left(\begin{array}{cc}
I&0\\
0&V_{1}^T
\end{array}\right)Q_{1}^TBQ_{1}\left(\begin{array}{cc}
I&0\\
0&V_{1}
\end{array}\right)=\left(\begin{array}{ccc}
B_1&B_4&B_5\\
B_4^T&B_6&0\\
B_5^T&0&0
\end{array}\right)$, where $B_6$ is a nonsingular diagonal matrix and $(B_4~B_5)=B_2V_1$; Denote $Q_2:=\diag(I,V_{1})$
\STATE Set $\hat{A}=Q_2^T\bar{A}Q_2=\bar{A}$ and $\hat{B}:=Q_2^T\bar{B}Q_2=\left(\begin{array}{ccc}
B_1&B_4&B_5\\
B_4^T&B_6&0\\
B_5^T&0&0
\end{array}\right)$\IF {$B_5$ exists and $B_5\neq0$}
\RETURN ``not SD''\ELSE
\STATE Denote $Q_3:=\left(\begin{array}{ccc}
I_p&0_{p\times q}&\\
-B_6^{-1}B_4^T&I_q&\\
&&I_{n-p-q}
\end{array}\right)$; Further denote $\tilde{A}:=Q_3^T\hat{A}Q_3=\hat{A}=\bar{A}$,
$\tilde{B}:=Q_3^T\hat{B}Q_3=\left(\begin{array}{ccc}
B_1-B_4B_6^{-1}B_4^T&0&\\
0&B_6&\\
&&0
\end{array}\right)$
\IF {the Jordan normal form of $A_1^{-1}(B_1-B_4B_6^{-1}B_4^T)$ is a diagonal matrix, i.e., $V_2^{-1}A_1^{-1}(B_1-B_4B_6^{-1}B_4^T)V_2=\diag(\lambda_1I_{n_1},\ldots,\lambda_tI_{n_t}),$
where $V_2$ is nonsingular}
\STATE  Find $R_k$, $k=1,\ldots,t$, which is a spectral decomposition matrix of the $k$th diagonal block of $V_2^{T}A_1V_2$; Denote
 $R:=\diag(R_1,\ldots,R_{t})$, $Q_4:=\diag(V_2R,I_{q+r})$ and $P:=Q_1Q_2Q_3Q_4$
\RETURN diagonal matrices $Q_4^T\tilde{A}Q_4$ and $Q_4^T\tilde{B}Q_4$ and  congruent matrix $P$
\ELSE
\RETURN``not SD''
\ENDIF
\ENDIF
\end{algorithmic}
\end{algorithm}

We next demonstrate the computational procedure to find the congruent matrix in checking whether a pair of matrices are SD via an illustrative example. The notations in the example follow the ones in Algorithm 1.

\textbf{Example 1}

Let $A:=\left(\begin{array}{cccccc}
1& 0& 0& 0& 0& 0\\
0& 4& 0& 0& 0& 0\\
0& 0& 9& 0& 0& 0\\
0& 0& 0& 0& 0& 0\\
0& 0& 0& 0& 0& 0\\
0& 0& 0& 0& 0& 0
\end{array}\right)$ and
$B:=\left(\begin{array}{cccccc}
1& 2& 0& 0& 3& 0\\
2& 5& 1& 0& 0& 0\\
0& 1& 7& 0& 0& 0\\
0& 0& 0& 2& 2& 0\\
3& 0& 0& 2& 5& 0\\
0& 0& 0& 0& 0& 0
\end{array}\right)$.
Applying the following congruent matrix which is derived in the appendix,
$$P:=\left(\begin{array}{cccccc}
3.98 & 194.23  & -0.21&0&0&0\\
    7.09 & -27.55  & -0.29&0&0&0\\
    1 &   1  &  1   &          0 &           0        &      0\\
3.98 &  194.22&   -0.21& -0.45      &      -0.89         &      0\\
-3.98&  -194.22&   0.21&   -0.89    &         0.45       &       0\\
0    &   0     &     0 &         0  &             0      & 1
\end{array}\right),$$
we obtain two diagonal matrices, $$P^TAP=\diag(226,40772,9,0,0,0),$$ and $$P^TBP=\diag(354,-93111,7,6,1,0).$$
\begin{rmk}
Computational methods for obtaining the Jordan normal form that are of complexity $O(n^3)$ can be found in \cite{BD} and \cite{GW}. One computational problem in computing the SD form for two matrices is the unstability of the numerical methods for Jordan normal form, see Chapter
7 in \cite{MC}, \cite{GW} and  \cite{K}. As a small perturbation in a problem setting may cause significant change in the formation of the Jordan blocks, it becomes difficult to deal with large-scale problems, see \cite{BD}.
 Our situation, however, is special, as we only need to calculate the Jordan normal form for real eigenvalues and we stop when the block size of any Jordan block is larger than 1 (implying that the two matrices are not SD). On the other hand, several symbolic computational methods have been developed to calculate the exact Jordan normal form (Please refer to \cite{LZW} and \cite{RV}).
\end{rmk}

\subsection{Applications to the generalized trust region subproblem and its variants}
Recall that the TRS is a special QCQP problem with one quadratic constraint:
\begin{eqnarray*}
{\rm (TRS)}~~~
&  \min  &\frac{1}{2}x^TBx+a^Tx \\
&  {\rm s.t.}&\Vert x\Vert_2 \leq1,
\end{eqnarray*}
where $B$ is an $n\times n$ symmetric matrix and $a \in \Re^n$. Numerous solution methods for problem formulation (TRS) have been developed in the literature, see for example, \cite{more1983computing} and \cite{rendl1997semidefinite}.

 We consider the following  extension of (TRS), which is called the generalized trust region subproblem (GTRS) in the literature:
\begin{eqnarray*}
{\rm (GTRS)}~~~
& \min  & \frac{1}{2}x^TBx+a^Tx \\
& {\rm s.t.}& \frac{1}{2}x^TAx+c^Tx+ \frac{1}{2}d \leq 0,
\end{eqnarray*}
where $A$ and $B\in S^n$ are $n\times n$ symmetric matrices, and $a,c,x \in \Re^n$ and $d\in \Re$.
Sturm and Zhang \cite{St2} prove the equivalence between the primal problem and its SDP relaxation for the GTRS. However, the state-of-the-art of SDP solvers do not support effective implementation in solving large scale GTRSs. Several other methods have been proposed for solving large scale instances of GTRSs under different conditions, see \cite{B1}, \cite{M} and \cite{S}.

The SD condition developed above in Section 2.1 could find its immediate applications in the GTRS. When the two matrices $A$ and $B$ are SD, the problem is equivalent to the following SOCP problem as showed in \cite{B1}:
\begin{eqnarray*}
& \min & \delta^Ty+\epsilon^Tx\\
& {\rm s.t.} &\alpha^Ty+\beta^Tx+ \frac{1}{2}d \leq0,\\
&&\frac{1}{2}x_i^2\leq y_i,~i=1,\ldots,n,
\end{eqnarray*}
where $\delta$, $\alpha$, $\epsilon$, $\beta \in \Re^n$, $\delta=\Dg(P^TBP)$, $\alpha=\Dg(P^TAP)$, $\epsilon=P^Ta$, $\beta=P^Tb$
and $P$ is the congruent matrix that makes both $P^TAP$ and $P^TBP$ diagonal. As SOCP reformulation can be solved much faster than SDP reformulation, it becomes possible to solve large scale GTRSs using SOCP reformulation when a problem is identified to be SD. Algorithm 1 actually enables us to verify whether a given instance of the GTRS is equivalent to its SOCP relaxation. Note that if the inequality in the GTRS becomes equality, the above analysis still holds true, see \cite{B1}.

Problem (TRS) is always SD, as the Jordan normal form of $I^{-1}B=B$ is a diagonal matrix (due to that the Jordan normal form of the symmetric matrix $B$ is diagonal). Thus, problem (TRS) is always equivalent to its SOCP relaxation.

When $B-\lambda A\succ0$ holds for some $\lambda\in \Re,$
  this instance of (GTRS) is classified as a regular case \cite{S}.  It is clear that all regular cases are SD, since $A$ and $B$ under the above regularity condition are SD as discussed in Section 1. We thus claim  that any regular case of the GTRS has an equivalent SOCP relaxation. We can further conclude that the applicability of the SD approach is wider than the existing methods based on regularity conditions.

 SD condition can also be applied to solve the following interval bounded GTRS:
\begin{eqnarray*}
{\rm (IGTRS)}~~~
&  \min& \frac{1}{2}x^TBx+a^Tx \\
&  {\rm s.t.}& l\leq\frac{1}{2}x^TAx+c^Tx \leq u,
\end{eqnarray*}
where lower and upper bounds $l$ and $u\in \Re$ are chosen such that $-\infty<l<u<+\infty$.
For regular cases of (IGTRS), \cite{W}  extends the results in \cite{S} to the interval bounded GTRS. Adopting Algorithm 1 developed in this paper
to identify whether a given instance of (IGTRS) is SD, we can verify whether the problem can be reduced to an SOCP problem based on the the results in \cite{B1} for QCQP with two quadratic constraints, which will be discussed in Section 4.
We emphasize that the regularity condition implies SD. Thus, the applicability of our approach is wider than the results in \cite{W}.

%

\section{Simultaneous diagonalization for finite number of matrices}

We now develop a method of checking simultaneous diagonalization for $m$ symmetric matrices $A_1,A_2,\ldots,A_m$, which have a semi-definite matrix pencil, i.e.,
there exists $\lambda\in \Re^m$ with $\lambda\neq0$, such that $\lambda_1A_1+\lambda_2A_2+\cdots+\lambda_mA_m\succeq0$. Note that this condition is always assumed in the SDP relaxation for a QCQP problem; Otherwise the SDP relaxation is unbounded from below, see \cite{KS,St}.

It is well known from \cite{HJ} that, if two matrices $A$ and $B$ commute, they are SD. It is interesting to extend this sufficient condition to a necessary and sufficient condition. The results in the following lemma have been mentioned in \cite{B1}. We provide its detailed proof in the following, which will be used to prove its extension, Theorem \ref{Theorem3-2}.
\begin{lemma}\label{lemma1}
Suppose that $A$ and $B$ are two symmetric $n\times n$ matrices, and $I$ is the $n\times n$ identity matrix. Then $A$, $B$ and $I$ are SD via an orthogonal congruent matrix if and only if $A$ and $B$ commute.
\end{lemma}
\begin{proof}
``$\Rightarrow$'' part: If $A$, $B$ and $I$ are SD, then there exists a nonsingular matrix $P$ such that $P^TAP$, $P^TBP$ and $P^TIP$ are diagonal.
Since $P^TIP$ is positive definite and diagonal, so there exists a diagonal matrix $Q$ such that $Q^TP^TIPQ=I$.
Let $U=PQ$, then $U^TAU$, $U^TBU$ and $U^TIU$ are all diagonal and $U^TIU=I$.
So $U$ is an orthogonal matrix. Furthermore, $U^TAU$ and $U^TBU$ are diagonal $\Rightarrow$ $U^TAUU^TBU=U^TBUU^TAU$ $\Rightarrow$ $AB=BA$, i.e., $A$ and $B$ commute.

``$\Leftarrow$'' part: Let $P$ be the spectral decomposition matrix for $A$ such that $P^TP=I$ and $P^TAP$ is diagonal. If $A$ and $B$ commute,
i.e., $AB=BA$, then by $PP^T=I$ we obtain $P^TAPP^TBP=P^TBPP^TAP$.
We can always assume $P$ to be a nonsingular matrix such that $P^TAP=\diag(\lambda_1I_{n_1},\lambda_2I_{n_2},\ldots,\lambda_kI_{n_k})$, where $k$ $\leq$ $n$, $n_k$ is the algebraic multiplicity of $\lambda_i$.
So, $P^TAPP^TBP=P^TBPP^TAP$ $\Rightarrow$ $P^TBP=\diag(B_1,B_2,\ldots,B_k)$ is a block diagonal matrix and $\dim(B_i)=\dim(I_{n_i})$, $\forall1\leq i\leq k$.
Using    the spectral decomposition to all $B_i$, we obtain $k$ diagonal matrices $Q_i^TB_iQ_i$, where  $Q_i^TQ_i=I_{n_i}$, and $\dim(Q_i)=\dim(I_{n_i})$, $\forall1\leq i\leq k$. By denoting $Q=\diag(Q_1,Q_2,\ldots,Q_k)$,  $Q^TP^TBPQ$ and $Q^TP^TAPQ$ are both diagonal matrices.
Besides, $Q^TP^TIPQ=I$. Thus, $A$, $B$ and $I$ are SD.
\end{proof}

One by-product of Lemma \ref{lemma1} is its applicability in solving the following CDT problem \cite{CDT}:
\begin{eqnarray*}
&  \min & \frac{1}{2}x^TBx+a^Tx \\
&  {\rm s.t.} &\parallel x\parallel_2 \leq1,\\
&&\parallel A^Tx+c\parallel_2 \leq1.
\end{eqnarray*}
Note that we can view the matrix in the quadratic term of the first inequality in the above CDT problem as $I$. We can then use Lemma \ref{lemma1}  to identify if the three matrices are SD, i.e., if $B$ and $AA^T$ commute. If so, we can then apply the method in \cite{B1} to solve the problem if one of the KKT multipliers is $0$.

The following theorem shows the SD condition for the identity matrix $I$ and other $m$ symmetric matrices when the congruent matrix is orthogonal, which is a classical result in linear algebra, see \cite{HJ}. For the sake of completeness, we give a brief proof here.
\begin{theorem}\label{Theorem3-2}
Matrices $I,A_1,A_2,\ldots,A_m$ are SD if and only if $A_i$ commutes with $A_j$, $\forall$$i,j$ = 1,2,$\ldots$, $m$, $i\neq j$.
\end{theorem}
\begin{proof}
Note that $I,A_1,A_2,\ldots,A_m$ are SD if and only if there exists a matrix $P$ such that $P^TP$ = $I$, and $P^TA_1P,\cdots,P^TA_mP$ are all diagonal matrices. Thus, $P$ is orthogonal.

We first prove the ``$\Rightarrow$'' part. Since there is an orthogonal matrix $P$, which makes $P^TA_iP$, $i$ = 1,$\ldots$, $m$, all diagonal, $P^TA_iPP^TA_jP=P^TA_jPP^TA_iP$. So $P^TA_iA_jP=P^TA_jA_iP$, and, thus, $A_iA_j=A_jA_i$, $i, j$ = 1, $\ldots$, $m$, $i$ $\neq$ $j$.

Next we prove the ``$\Leftarrow$'' part. As the case where $m = 2$ has already been proved in Lemma \ref{lemma1}, we will use the induction principle to prove the general case.
Suppose that $Q_1^TA_1Q_1$ is the spectral decomposition of $A_1$ such that $Q_1^TQ_1$ = $I$ and $Q_1^TA_1Q_1=\diag(\lambda_1I_{n_1},\lambda_2I_{n_2},\ldots,\lambda_kI_{n_k})$,
with $1\leq$ $k$ $\leq$ $n$ and $\lambda_i$ $\neq$ $\lambda_j$, if $i$ $\neq$ $j$, $1\leq i,j\leq k$.
Then,  $A_1A_i=A_iA_1$ $\Rightarrow$  $Q_1^TA_1Q_1Q_1^TA_iQ_1=Q_1^TA_iQ_1Q_1^TA_1Q_1$ $\Rightarrow$ $\tilde{A}_i := Q_1^TA_iQ_1$ is a block diagonal matrix in the form of
$\diag(\tilde{A}_{i_1},\tilde{A}_{i_2},$ $\cdots,\tilde{A}_{i_k})$, with $\dim(\tilde{A}_{i_l})=\dim(I_{n_{l}})$, $i=1,\ldots,m$, $l=1, 2, \ldots, k$. Furthermore,
$$A_iA_j=A_jA_i\Rightarrow Q_1^TA_iQ_1Q_1^TA_jQ_1=Q_1^TA_jQ_1Q_1^TA_iQ_1.$$
So $\tilde{A}_{i_l}$ commutes with $\tilde{A}_{j_l}$, $\forall2\leq i,j\leq m,~ l=1,2,\cdots,k$. Now for every $l=1,2,\cdots,k$, the case reduces to prove if  $m-1$ matrices $\tilde{A}_{i_l}$, 2 $\leq$ $i$ $\leq$ $m$, commute with each other, then they are SD by some orthogonal matrix $R_l$. Note that the congruent matrix, $R=\diag(R_1,\ldots,R_k)$ has the same block structure as $\tilde{A}_i$. According to the induction principle
and Lemma \ref{lemma1}, which represents the case $m=2$, we complete the proof.
\end{proof}

Theorem \ref{Theorem3-2} also gives a recursion procedure to identify if $I,A_1,A_2,\ldots,A_m$ are SD.

By applying  Theorem \ref{Theorem3-2}, we have the following theorem.\begin{theorem}\label{Theorem3-3}
If there exists $\lambda\in \Re^m$ such that $\lambda_1A_1+\lambda_2A_2+\cdots+\lambda_mA_m\succ0$, where, without loss of generality,
$\lambda_m$ is assumed not to be zero, then $A_1,A_2,\ldots, A_m$ are SD if and only if $P^TA_iP$ commute with $P^TA_jP$,
$\forall i\neq j$, $1\leq$ $i,~j$ $\leq$ $m-1$, where P is any nonsingular matrix that makes $P^T(\lambda_1A_1+\lambda_2A_2+\cdots+\lambda_mA_m)P=I.$
\end{theorem}

\begin{proof}
The matrix $\lambda_1A_1+\lambda_2A_2+\cdots+\lambda_mA_m$ is positive definite and thus there exists an orthogonal  matrix $U$ such that
$D=U^T( \lambda_1A_1+\lambda_2A_2+\cdots+\lambda_mA_m)U$ is a diagonal matrix. Denote $D^\frac{1}{2}=\diag(D_{11}^\frac{1}{2},D_{22}^\frac{1}{2},\ldots,D_{nn}^\frac{1}{2})$,
and then $D^\frac{1}{2}U^T( \lambda_1A_1+\lambda_2A_2+\cdots+\lambda_mA_m)UD^\frac{1}{2}=I$.
Note that $A_1,\ldots,A_{m-1}, \lambda_1A_1+\lambda_2A_2+\cdots+\lambda_mA_m$ are SD if and only if
$P^TA_1P,\ldots, P^TA_{m-1}P,P^T(\lambda_1A_1+\lambda_2A_2+\cdots+\lambda_mA_m)P$ are SD, and from Theorem \ref{Theorem3-2}, we conclude that
$P^TA_1P,\ldots,P^TA_{m-1}P, P^T(\lambda_1A_1+\lambda_2A_2+\cdots+\lambda_mA_m)P$ are SD if and only if $P^TA_iP$ commutes  with $P^TA_jP$, $\forall i\neq j$,
$1\leq$ $i,~j$ $\leq$ $m-1$, where $P=UD^{-\frac{1}{2}}$ and $P^T(\lambda_1A_1+\lambda_2A_2+\cdots+\lambda_mA_m)P=I.$
So we only need to prove that  $A_1,A_2,\cdots, A_m$ are SD $\Leftrightarrow$ $A_1,A_2,\ldots, A_{m-1},\lambda_1A_1+\lambda_2A_2+\cdots+\lambda_mA_m$ are SD.

We first prove the ``$\Rightarrow$'' part: If $A_1,A_2,\ldots, A_m$ are SD,  then  there exist a nonsingular matrix $P$ such that $P^TA_1P,~P^TA_2P,~\ldots,~P^TA_mP$ are all diagonal matrices.
Thus,  $\lambda_1P^TA_1P+\lambda_2P^TA_2P+\cdots+\lambda_mP^TA_mP$ is a diagonal matrix. We therefore obtain that
$A_1,A_2,\ldots,~A_{m-1},\lambda_1A_1+\lambda_2A_2+\cdots+\lambda_mA_m$ are SD.

 We prove next the ``$\Leftarrow$'' part: If $A_1,A_2,\ldots, A_{m-1}, \lambda_1A_1+\lambda_2A_2+\cdots+\lambda_mA_m$ are SD,
 then there exists a nonsingular matrix $P$ such that $P^TA_1P,~P^TA_2P,~\ldots,~\\P^TA_{m-1}P,P^T(\lambda_1A_1+\lambda_2A_2+\cdots+\lambda_mA_m)P$ are all diagonal. Since $\lambda_m$ is not zero, $P^TA_mP=[P^T(\lambda_1A_1+\lambda_2A_2+\cdots+\lambda_mA_m)P-\sum_{i=1}^{m-1}\lambda_i P^TA_iP]/\lambda_m$ is also diagonal.
\end{proof}

For the semi-definite matrix pencil case, we have the following results.
\begin{theorem}\label{Theorem3-4}
Suppose that $A_1,A_2,\ldots,A_m$ and $B$ are  $n\times n$ symmetric matrices with the following forms,
$$A_i:=\left(\begin{array}{cc}
A_i^1&A_i^2\\
(A_i^2)^T&A_i^3\\\end{array}\right), {\rm~and~} B:=\left(\begin{array}{cc}
I_p&\\
&0\\\end{array}\right),$$ where $\dim A_i^1=\dim I_p=p<n$, $A_i^3$ is diagonal, $\forall1\leq i\leq m$, and $A_1^3$ is nonsingular.
Then $A_1,A_2,\ldots,A_m$ and $B$  are SD if and only if the following two conditions are satisfied:\\
1. $A_i^2=A_1^2(A_1^3)^{-1}A_i^3$, $i=2,\ldots,m$.\\
2. $A_i^1-A_i^2(A_1^3)^{-1}(A_1^2)^T$, $i=1,\ldots,m$, mutually commute.
\end{theorem}
\begin{proof}
The ``$\Rightarrow$'' part: If  $A_1,A_2,\ldots,A_m$ and $B$ are SD,  there exists a nonsingular matrix $P$ such that $P^TA_1P,P^TA_2P,\ldots,P^TA_mP$ and $P^TBP$ are all diagonal,
and we can assume, without loss of generality, $P^TBP=B$.
As $P^TBP=B$ if and only if
$$\left(\begin{array}{cc}
P_1^T&P_3^T\\
P_2^T&P_4^T\\\end{array}\right)\left(\begin{array}{cc}
I_p&\\
&0\\\end{array}\right)\left(\begin{array}{cc}
P_1&P_2\\
P_3&P_4\\\end{array}\right)=\left(\begin{array}{cc}
P_1^TP_1&P_1^TP_2\\
P_2^TP_1&P_2^TP_2\\\end{array}\right)=\left(\begin{array}{cc}
I_p&\\
&0\\\end{array}\right),$$
we conclude $P_1$ is orthogonal and $P_2=0$.

Furthermore, we have
\begin{eqnarray*}
P^TA_iP
&=&\left(\begin{array}{cc}
P_1^T&P_3^T\\
P_2^T&P_4^T\\\end{array}\right)\left(\begin{array}{cc}
A_i^1&A_i^2\\
(A_i^2)^T&A_i^3\\\end{array}\right)\left(\begin{array}{cc}
P_1&P_2\\
P_3&P_4\\\end{array}\right)\\
&=&\left(\begin{array}{cc}
P_1^TA_i^1P_1+P_3^TA_i^3P_3+P_1^TA_i^2P_3+P_3^T(A_i^2)^TP_1&P_1^TA_i^2P_4+P_3^TA_i^3P_4\\
P_4^T(A_i^2)^TP_1+P_4^TA_i^3P_3&P_4^TA_i^3P_4\\\end{array}\right).
\end{eqnarray*}

Since $P^TA_iP,~ i=1, \ldots, m$,  are diagonal, $P_4^T(A_i^2)^TP_1+P_4^TA_i^3P_3=0$ holds true for $i=1, \ldots, m$. As $P_4$ is nonsingular, we have $(A_i^2)^TP_1+A_i^3P_3=0$ for $i=1, \ldots, m$.
Solving $(A_1^2)^TP_1+A_1^3P_3$ = $0$ yields $P_3=-(A_1^3)^{-1}(A_1^2)^TP_1$. Substituting $P_3=-(A_1^3)^{-1}(A_1^2)^TP_1$ into $(A_i^2)^TP_1+A_i^3P_3=0$ and by the nonsingularity of $P_1$,
we further have $A_i^3(A_1^3)^{-1}(A_1^2)^T=(A_i^2)^T$, $i=2, \ldots, m$.

 Substituting $P_3=-(A_1^3)^{-1}(A_1^2)^TP_1$ into the (1,1)th block of $P^TA_iP$, $P_1^TA_i^1P_1+P_3^TA_i^3P_3+P_1^TA_i^2P_3+P_3^T(A_i^2)^TP_1$  gives rise to $P_1^T(A_1^1-A_1^2(A_1^3)^{-1}(A_1^2)^T)P_1$ when $i$ = $1$, and
\begin{eqnarray*}
&&P_1^T(A_i^1+A_1^2(A_1^3)^{-1}A_i^3(A_1^3)^{-1}(A_1^2)^T-A_i^2(A_1^3)^{-1}(A_1^2)^T-A_1^2(A_1^3)^{-1}(A_i^2)^T)P_1\\
&=&P_1^T(A_i^1-A_i^2(A_1^3)^{-1}(A_1^2)^T)P_1,
\end{eqnarray*} for $i=2,\ldots,m$.
Since the (1,1)th blocks of $P^TA_iP,~ i=1\ldots,m,$ and $B$ are all diagonal matrices and $P_1$ is orthogonal, we conclude
that $A_i^1-A_i^2(A_1^3)^{-1}(A_1^2)^T,~i=1,2,\ldots,m$, mutually commute.

The ``$\Leftarrow$'' part: If conditions 1 and 2 hold, then $P^TA_iP,~ i=1\ldots,m,$ and $P^TBP$ are all diagonal by choosing$$P=\left(\begin{array}{cc}
P_1&0\\
-(A_1^3)^{-1}(A_1^2)^TP_1&I_{n-p}
\end{array}\right), $$
where $P_1$ is the orthogonal matrix such that $P_1^T(A_i^1-A_i^2(A_1^3)^{-1}(A_1^2)^T)P_1$, $i=1,\ldots,m$, are all diagonal. Note that the existence of
$P_1$ is due to condition 2 and Theorem \ref{Theorem3-2}.
\end{proof}

Now we extend the result in Theorem \ref{Theorem3-4}, in which $\lambda_1A_1+\lambda_2A_2+\cdots+\lambda_mA_m\succ0$ holds true, to situations where only the following assumption holds.

\begin{asmp}\label{Strict-Semi}
For $m$ symmetric $n\times n$ matrices, $A_1,A_2,\ldots,A_m$, there exists $\lambda \in \Re^m$ and $\lambda\neq 0$, such that $\lambda_1A_1+\lambda_2A_2+\cdots+\lambda_mA_m\succeq0$,
but there does not exist $\lambda\in R^m$ such that $\lambda_1A_1+\lambda_2A_2+\cdots+\lambda_mA_m\succ0$.
Without loss of generality, we assume $\lambda_m\neq0$.
\end{asmp}

For a given set of $A_i$, $i$ = 1, $\ldots$, $m$, we can always use an SDP formulation to verify whether there exist
$\lambda\neq0$, which satisfies Assumption \ref{Strict-Semi}. If  the following
SDP is feasible, then there exists such a $\lambda\neq0$ because otherwise $tr(\sum_{i=1}^m\lambda_iA_i)=0$,
\begin{eqnarray*}
&\min &0\\
&{\rm s.t.}&\sum_{i=1}^m\lambda_iA_i\succeq0,\\
&&tr(\sum_{i=1}^m\lambda_iA_i)\geq1.
\end{eqnarray*}

When Assumption \ref{Strict-Semi} holds,  we can find a nonsingular matrix $Q_1$ and the corresponding $  \lambda\in \Re^m$ such that $\mathcal{A}_m:=Q_1^T(\lambda_1A_1+\lambda_2A_2+\cdots+\lambda_mA_m)Q_{1}=\left(\begin{array}{cc}
I_p&\\
&0\end{array}\right)$, and
\begin{equation}\label{mathcal-A}
\mathcal{A}_i:=Q_{1}^TA_iQ_{1}=\left(\begin{array}{ccc}
\mathcal{A}_i^1&\mathcal{A}_i^2\\
(\mathcal{A}_i^2)^T&\mathcal{A}_i^3\end{array}\right),~
 i=1, 2, \ldots, m-1,
 \end{equation}
 where $\dim \mathcal{A}_i^1=\dim I_p=p<n$.

Furthermore, if all $\mathcal{A}_i^3$, $i$ = 1, $\ldots$, $m$, are SD, then, by rearranging the common 0s to the lower right corner of the matrix, there exists a nonsingular matrix $Q_{2}:=\diag(I_p,V)$ such that
\begin{equation}\label{mathbf-Am}
\mathbf{A}_m:=Q_{2}^T\mathcal{A}_mQ_{2}=\left(\begin{array}{cc}
I_p&\\
&0\end{array}\right),
\end{equation} and
\begin{equation}\label{mathbf-A}
\mathbf{A}_i:   =Q_{2}^T\mathcal{A}_iQ_{2}=\left(\begin{array}{ccc}
\mathbf{A}_i^1&\mathbf{A}_i^2&\mathbf{A}_i^4\\
(\mathbf{A}_i^2)^T&\mathbf{A}_i^3&0\\
(\mathbf{A}_i^4)^T&0&0\end{array}\right),
\end{equation}
where $\mathbf{A}_i^1=\mathcal{A}_i^1,$ $\mathbf{A}_i^3, i=1,2,\ldots,m-1$, are all diagonal matrices and do not have common 0s in the same positions.

For any diagonal matrices $D$ and $E$, denote $\supp(D):=\{i~|~D_{ii}\neq0\}$ and $\supp(D)\cup \supp(E):=\{i~|~D_{ii}\neq0~{\rm or ~} E_{ii}\neq0 \}$.
\begin{lemma}\label{alpha}
For $k ~(k\geq2)$ $n\times n$ nonzero diagonal matrices $D^1,D^2,\ldots,D^k$, if there exists no common 0s in the same position, then the following procedure will find $\mu_i\in \Re, i=1,\ldots,k$ such that $\sum_{i=1}^k\mu_iD^i$ is nonsingular:

Step 1. Let $D=D^1$, $\mu_1=1$, and $\mu_i=0$, for $i=2,3,\ldots,n$, $j=1$;

Step 2. Let $D^*=D+\mu_{j+1}D^{j+1}$, where $\mu_{j+1}=\frac{s}{n}, s\in\{0,1,\ldots,n\}$ with  $s$ being chosen such that $D^*=D+\mu_{j+1}D^{j+1}$ and $\supp(D^*)=\supp(D)\cup \supp(D^{j+1})$;

Step 3. Let $D=D^*$, j=j+1; if $D$ is nonsingular or $j=n$, STOP and output $D$; Else, go to Step 2.
\end{lemma}
\begin{proof}
We only need to prove that in Step 2 we can find a suitable $s$.
The cardinality of $\supp(D^{j+1})$ is at most $n$ and that of $\supp(D)$ is at most $n-1$ (or we have already found a nonsingular matrix $D$). So if  $\supp(D^{j+1})=\supp(D)$, let $s=0$, and we have $\supp(D^*)=\supp(D)\cup\supp(D^{j+1})$.
Otherwise, for any $i\in \supp(D^{j+1})$, varying $s$ from 0 to $n$ will generate $n+1$ different values for $D_{ii}+\frac{s}{n}D^{j+1}_{ii}$, and at most one $s$ can lead $D_{ii}+\frac{s}{n}D^{j+1}_{ii}=0$.
 But as the cardinality of $\supp(D^{j+1})$ $\leq n$, we can always find an $s$ such that $D_{ii}+\mu_iD^{j+1}_{ii}\neq0$ for all $i\in \supp(D)\cup\supp(D^{j+1})$. One way to find such an $s$ is to test $s$ from $0$ to $n$ to find out the value of $s$ such that $\supp(D+\frac{s}{n}D^{j+1})=\supp(D)\cup \supp(D^{j+1})$.
\end{proof}

Denote
\begin{equation}\label{D3.5}
\mathbf{D}=\sum_{i=1}^{m-1} \mu_i \mathbf{A}_i=\left(\begin{array}{ccc}
\mathbf{D}_1&\mathbf{D}_2&\mathbf{D}_4\\
(\mathbf{D}_2)^T&\mathbf{D}_3&0\\
(\mathbf{D}_4)^T&0&0\end{array}\right),
\end{equation} where $\mu_i, i=1,2,\ldots,m-1$, are chosen, via the procedure in Lemma \ref{alpha}, such that $\mathbf{D}_3$ is nonsingular.
Then we have the following theorem for a semi-definite matrix pencil.

\begin{theorem}\label{3.5}
Under Assumption \ref{Strict-Semi}, $A_1,A_2,\ldots,A_m$ are SD $\Leftrightarrow$
$A_1,A_2,\ldots,\\A_{m-1}$ and $\tilde{A}_m:=\lambda_1A_1+\lambda_2A_2+\cdots+\lambda_mA_m\succeq0$ are SD $\Leftrightarrow$
$\mathcal{A}_i^3$ (defined in (\ref{mathcal-A})), $i=1,2,\ldots,m$, are SD, and the following conditions are also satisfied:\\
1. $\mathbf{A}_i^4=0$, $i=1,2,\ldots,m-1$,\\
2. $\mathbf{A}_i^2=\mathbf{D}_2 \mathbf{D}_3^{-1}\mathbf{A}_i^3,~i=1,2,\ldots,m-1$,\\
3. $\mathbf{A}_i^1-\mathbf{A}_i^2\mathbf{D}_3^{-1}\mathbf{D}_2^T,~i=1,2,\ldots,m-1$, mutually commute, where $\mathbf{A}_i^1$,
$\mathbf{A}_i^2$, $\mathbf{A}_i^3$ and $\mathbf{A}_i^4$ are defined in (\ref{mathbf-A}), and $\mathbf{D}$ is defined in (\ref{D3.5}).
\end{theorem}
\begin{proof}
The first ``$\Leftrightarrow$'' can be proved by the same technique with the proof for Theorem \ref{Theorem3-3}.

We next prove the second ``$\Leftrightarrow$''.

The ``$\Rightarrow$'' part: $A_1,A_2,\ldots,A_{m-1}$ and $\tilde{A}_m$ are SD $\Rightarrow$ $\mathcal{A}_1,\mathcal{A}_2,\ldots,\mathcal{A}_m$
are SD (note that $\mathcal{A}_m=Q_1^T\tilde{A}_mQ_1$ by definition). From (\ref{mathcal-A}), (\ref{mathbf-Am})  and (\ref{mathbf-A}),
we know that there exists a nonsingular matrix $Q_{2}$ such that $Q_{2}^T\mathcal{A}_i Q_{2}, i=1,2,\ldots,m$, are SD and have the form in  (\ref{mathbf-Am})  and (\ref{mathbf-A}),
i.e.,$$\mathbf{A}_m:=Q_{2}^T\mathcal{A}_mQ_{2}=\left(\begin{array}{cc}
I_p&\\
&0\\\end{array}\right),$$ and $$\mathbf{A}_i:=Q_{2}^T\mathcal{A}_iQ_{2}=\left(\begin{array}{ccc}
\mathbf{A}_i^1&\mathbf{A}_i^2&\mathbf{A}_i^4\\
(\mathbf{A}_i^2)^T&\mathbf{A}_i^3&0\\
(\mathbf{A}_i^4)^T&0&0\end{array}\right).$$

Since $\mathcal{A}_i$ and $\mathcal{A}_m$ are SD, $i=1,2,\ldots,m-1$
$\Leftrightarrow$ $\mathbf{A}_i$ and $\mathbf{A}_m$ are SD, $i=1,2,\ldots,m-1$
$\Leftrightarrow$ $\mathbf{D},~\mathbf{A}_i$, $i=2,3,\ldots,m-1$, and $\mathbf{A}_m$ are SD.
From Theorem \ref{Two-SD} we know $\mathbf{D}_4=0$, $\mathbf{A}_i^4=0$, $i=2,3,\ldots,m-1$, so $\mathbf{A}_i^4=0$, $i=1,2,\ldots,m-1$.
Then from the proof in Theorem \ref{Theorem3-4}, we can conclude that conditions 2 and 3 are also satisfied.

The ``$\Leftarrow$'' part: Note that $Q_1^{T}A_iQ_{1}=\mathcal{A}_i$, $i=1,\ldots,m-1$.
Then if $\mathcal{A}_i^3$ are SD, $i=1,2,\ldots,m$, there exists a nonsingular matrix $Q_2$ such that (\ref{mathbf-Am}) and (\ref{mathbf-A}) holds. Further denote
\begin{equation}\label{conmatrix}
P:=\left(\begin{array}{ccc}
P_1&0&0\\
-(\mathbf{D}_3)^{-1}(\mathbf{D}_2)^TP_1&I_q&0\\
0&0&I_{n-p-q}
\end{array}\right),
\end{equation}
where $P_1$ is the orthogonal matrix such that $P_1^T(\mathbf{A}_i^1-\mathbf{A}_i^2(\mathbf{D}_3)^{-1}(\mathbf{D}_2)^T)P_1$, $i=1,\ldots,m-1$, are all diagonal,  and $q=\dim(\mathbf{A}_1^{3})$.
Then one can check that
\begin{eqnarray*}
&&P^T\mathbf{A}_iP\\
&=&\left(\begin{array}{ccc}
\begin{array}{l}P_1^T(\mathbf{A}_i^1-\mathbf{A}_i^2\mathbf{D}_3^{-1}\mathbf{D}_2^T-\mathbf{D}_2\mathbf{D}_3^{-1}(\mathbf{A}_i^2)^T\\
+\mathbf{D}_2\mathbf{D}_3^{-1}\mathbf{A}_i^3\mathbf{D}_3^{-1}\mathbf{D}_2^T)P_1\end{array}&P_1^T(\mathbf{A}_i^2-\mathbf{D}_2\mathbf{D}_3^{-1}\mathbf{A}_i^3)&0\\
((\mathbf{A}_i^2)^T-\mathbf{A}_i^3\mathbf{D}_3^{-1}\mathbf{D}_2^T)P_1&\mathbf{A}_i^3&0\\
0&0&I_{n-p-q}\end{array}\right)\\
&=&\left(\begin{array}{ccc}
P_1^T(\mathbf{A}_i^1-\mathbf{A}_i^2\mathbf{D}_3^{-1}\mathbf{D}_2^T)P_1&&\\
&\mathbf{A}_i^3&\\
&&I_{n-p-q}\end{array}\right),
\end{eqnarray*}
$i=1\ldots,m,$ are all diagonal. Thus, $A_1,A_2,\ldots,A_{m-1}$ and $\tilde{A}_m:=\lambda_1A_1+\lambda_2A_2+\cdots+\lambda_mA_m\succeq0$ are SD via the congruent matrix $Q_1Q_2P.$
\end{proof}

Although we choose $\mu$ following the procedure in Lemma \ref{alpha}, the proof shows that Theorem \ref{3.5} always holds for every $\mu$ such that $\mathbf{D}_3$ is nonsingular.
Based on the above discussion, we develop a systematic procedure in Algorithm \ref{alg2} to verify whether the given $m$ matrices are SD.

\begin{algorithm}
\caption{Procedure to check whether $m$ matrices are SD}
\label{alg2}
\begin{algorithmic}[1]
\REQUIRE $m$ symmetric matrices $A_1, A_2, \ldots, A_m$, which satisfy that $\exists \lambda\in \Re^m,~\lambda\neq0,$ such that $\sum_{i=1}^m\lambda_iA_i\succeq0$
\STATE Find a nonsingular matrix $Q_1$ such that $\mathcal{A}_m=Q_1^T(\lambda_1A_1+\lambda_2A_2+\cdots+\lambda_mA_m)Q_1=\left(\begin{array}{cc}
I_p&\\
&0\end{array}\right)$, in which we assume $\lambda_m\neq0$, otherwise we can exchange the index $m$ with $k$ for some $\lambda_k\neq0$; Define
$
\mathcal{A}_i=Q_1^TA_iQ_1=\left(\begin{array}{cc}
\mathcal{A}_i^1&\mathcal{A}_i^2\\
(\mathcal{A}_i^2)^T&\mathcal{A}_i^3\end{array}\right),
 ~i=1, 2, \ldots, m-1
$

\IF {$\mathcal{A}_i^3, ~i=1,2,\ldots,m-1$, are not SD}
\RETURN``not SD''\ELSE
\STATE Find a congruent matrix $V$ that makes all $\mathcal{A}_i^{3}$ diagonal\footnotemark, and  $Q_2=\diag(I_p,V)$ that makes $Q_2^TQ_1^T\mathcal{A}_iQ_1Q_2=\left(\begin{array}{ccc}
\mathbf{A}_i^1&\mathbf{A}_i^2&\mathbf{A}_i^4\\
(\mathbf{A}_i^2)^T&\mathbf{A}_i^3&0\\
(\mathbf{A}_i^4)^T&0&0\end{array}\right)$,  $i=1,\ldots,m-1$

\STATE Find $\mu_i,~ i=1,2,\ldots,m-1$, via the procedure in Lemma \ref{alpha}, such that $\mathbf{D}_3:=\sum_{i=1}^{m-1} \mu_i \mathbf{A}_i^3$ is nonsingular, and denote $\mathbf{D}_2:=\sum_{i=1}^{m-1} \mu_i \mathbf{A}_i^2$
\IF {i) $\mathbf{A}_i^4=0$, $i=1,2,\ldots,m-1$, ii) $\mathbf{A}_i^2=\mathbf{D}_2 \mathbf{D}_3^{-1}\mathbf{A}_i^3,i=2,\ldots,m-1$, and
iii) $\mathbf{A}_i^1-\mathbf{A}_i^2\mathbf{D}_3^{-1}\mathbf{D}_2^T,~i=2,\ldots,m-1$, mutually commute}
\RETURN $U^TA_iU, i=1,\ldots,m$, and $U=Q_1Q_2P$, where $P$ is defined in (\ref{conmatrix})\ELSE
\RETURN ``not SD''
\ENDIF
\ENDIF
\end{algorithmic}
\end{algorithm}

\textbf{Example 2}
The following example shows a case in which three matrices do not mutually commute, but they are SD. Consider
$A=\left(\begin{array}{cc}
1&2\\
2&20\\\end{array}\right)$,
$B=\left(\begin{array}{cc}
-1&-2\\
-2&-28\\\end{array}\right)$, and
$C=\left(\begin{array}{cc}
3&6\\
6&-20\\\end{array}\right)$.
It is easy to check that $AB\neq BA$. Let $U=\left(\begin{array}{cc}
1&-0.5\\
0&0.25\\\end{array}\right)$. Then, we can show that $A$, $B$ and $C$ are SD, as $U^TAU=\left(\begin{array}{cc}
1&0\\
0&1\\\end{array}\right)$, $U^TBU=\left(\begin{array}{cc}
-1&0\\
0&-1.5\\\end{array}\right)$, and $U^TCU=\left(\begin{array}{cc}
3&0\\
0&-2\\\end{array}\right)$.

\section{Applications to quadratically constrained quadratic programming}

Consider the following QCQP problem:
\begin{eqnarray*}
{\rm (QP)}~~~
& \min & f^0(x) \\
&  {\rm s.t.} & f^i(x)\leq0,~i=1,\ldots,m,
\end{eqnarray*}
where $f^0(x)=\frac{1}{2}x^TA_0x+a_0^Tx$, and $f^i(x)=\frac{1}{2}x^TA_ix+a_i^Tx+\frac{1}{2}d_i$, $A_0,A_i\in S^n$, $a_0,a_i\in \Re^n$ and
$d_i\in\Re$, $i=1,\ldots,m$.

Let us consider first the homogeneous case where $a_i=0,~ i=0,1,\ldots,m$.
If $A_0,A_1,\ldots,A_m$ are SD with $\overline{A}_i$ being the diagonalized matrix of $A_i$, then $\rm(QP)$ is equivalent to the following problem:
\begin{eqnarray*}
{\rm (HQPSD)}~~~
& \min & \sum_{k=1}^n\frac{1}{2}\Dg(\bar{A}_0)_kx_k^2 \\
&  {\rm s.t.} &  \sum_{k=1}^n \frac{1}{2}\Dg(\bar{A}_i)_kx_k^2 +\frac{1}{2}d_i\leq0,~i=1,\ldots,m.
\end{eqnarray*}
\footnotetext{In the algorithm, we always rearrange the common 0s to the lower right corner of the $V^T\mathcal{A}_i^3V$, $i$ = 1, $\ldots$, $m-1$, via the congruent matrix $V$, and use a recursion to find the congruent matrix $V$ that makes all $\mathcal{A}_i^{3}$ diagonal.}

In fact, ($\rm HQPSD$) is equivalent to the following linear programming problem formulation:
\begin{eqnarray*}
&  \min & \sum_{k=1}^n\frac{1}{2}\Dg(\bar{A}_0)_k y_k \\
&  {\rm s.t.} &   \sum_{k=1}^n \frac{1}{2}\Dg(\bar{A}_i)_ky_k +\frac{1}{2}d_i\leq0,~i=1,\ldots,m,\\
&~&   y_k\geq0,~k=1,\ldots,n,
\end{eqnarray*}
which can be solved efficiently by several methods in the literature, e.g., the simplex algorithm.

Introducing $x_{n+1}=\pm1$ in nonhomogeneous (QP) gives rise to the following equivalent problem formulation:
\begin{eqnarray*}
{\rm (QP')}~~~
& \min &\frac{1}{2}x^TA_0x+a_0^Txx_{n+1} \\
& {\rm s.t.} & \frac{1}{2}x^TA_ix+a_i^Txx_{n+1}+\frac{1}{2}d_{i}x_{n+1}^2\leq0,~i=1,\ldots,m,\\
&&x_{n+1}^2=1.
\end{eqnarray*}
Denote $B_i=\left(\begin{array}{ccc}
A_i&a_i\\
a_i^T&d_i\end{array}\right), ~i=0, 1, \ldots, m$ (here $d_0=0$), and $B_{m+1}=\left(\begin{array}{ccc}
0_n&\\
&1\end{array}\right)$. We can rewrite (QP$'$) as the following homogeneous problem:
\begin{eqnarray*}
&\min&\frac{1}{2}x^TB_0x \\
&  {\rm s.t.} & \frac{1}{2}x^TB_ix\leq0,~i=1,\ldots,m,\\
&& x^TB_{m+1}x=1.
\end{eqnarray*}
We can then apply Algorithm \ref{alg2} to check the SD condition of the above problem. In an SD situation, a nonhomogeneous QCQP problem can be also reduced to an equivalent linear programming formulation, similar to the homogeneous case.

Furthermore, when $m=1$ or $2$, the above problem has exact relaxations in some cases as showed below.

For $m=1$, it is the GTRS, which possesses an exact relaxation when SD condition holds as we discussed in (TRS) and (GTRS) in Section 2.2.

For $m=2$, we have the following SOCP relaxation when the three matrices are SD:
\begin{eqnarray*}
{\rm(QP_2-SOCP)}~~~~~& \min& \delta^Ty+\epsilon^Tx\\
&  {\rm s.t.} &\alpha^Ty+\beta^Tx \leq \frac{1}{2}d_{1},\\
& &\eta^Ty+\theta^Tx \leq \frac{1}{2}d_{2},\\
&& \frac{1}{2}x_i^2\leq y_i,~i=1,\ldots,n,
\end{eqnarray*}
where $\delta,\epsilon,\alpha,\beta,\eta,\theta\in \Re^n$, $\delta=\Dg(P^TA_0P)$, $\alpha=\Dg(P^TA_1P)$, $\eta=\Dg(P^TA_2P)$,
$\epsilon=P^Ta_0$, $\beta=P^Ta_1$, $\theta=P^Ta_2$ and $P$ is the congruent matrix that makes $P^TA_iP$, $i=0,1,2$, all diagonal.
Ben-Tal and Hertog \cite{B1} demonstrate that if one of the KKT multipliers of the first two constraints in $\rm(QP_2-SOCP)$ is 0, then the SOCP relaxation is exact.
For the SOCP relaxation of problem (IGTRS), the two quadratic constraints are relaxed to
\begin{eqnarray*}
&&\alpha^Ty+\beta^Tx\leq u,\\&&-\alpha^Ty-\beta^Tx\leq -l,
\end{eqnarray*}
and one of the KKT multipliers must be 0, as the two boundary conditions cannot be binding simultaneously, i.e., it can not be $l=\alpha^Ty+\beta^Tx=u$, thus, satisfying Assumption 6 in \cite{B1} automatically.
\section{Conclusion}
In this paper, we have succeeded in providing complete answers to the open question on simultaneous diagonalization (SD) posted in \cite{H}. More specifically, we have identified a necessary and sufficient SD condition for any two real symmetric matrices and a necessary and sufficient SD condition for multiple (more than two) real symmetric matrices under the existence assumption of a semi-definite matrix pencil. Furthermore, we have demonstrated how SD can be utilized as a powerful instrument to verify the exactness of the SOCP relaxation  of QCQP, especially with one or two quadratic constraints, and to facilitate the solution process.  One of our future work is to find a necessary and sufficient SD condition for multiple matrices without the assumption of semi-definite matrix pencil and find more real-world applications of our SD procedure.

\Appendix\section{Derivation of the congruent matrix for Example 1}
In this appendix, we show how to derive the congruent matrix $P$ for Example 1 by applying Algorithm 1:
First note that $A$ is already in the form of $\rm diag(A_1,0)$ and thus we do not need to goto Line 1.
Denote $A_1:=\left(\begin{array}{ccc}
1& 0& 0\\
0& 4& 0\\
0& 0& 9\end{array}\right)$,
$B_1:=\left(\begin{array}{ccc}
1& 2& 0\\
2& 5& 1\\
0& 1& 7\end{array}\right)$,
 $B_2:=\left(\begin{array}{ccc}
0& 3& 0\\
0& 0& 0\\
0& 0& 0\end{array}\right)$, and
 $B_3:=\left(\begin{array}{ccc}
2& 2& 0\\
2& 5& 0\\
0& 0& 0\end{array}\right)$.
Next we goto Line 2.
Applying the spectral decomposition to $B_3$ yields $$V_{1}^TB_3V_{1}=\left(\begin{array}{ccc}
6& 0& 0\\
0& 1& 0\\
0& 0& 0\end{array}\right), {\rm~ where~}V_{1}:=\left(\begin{array}{ccc}
-0.4472&  -0.8944& 0\\
 -0.8944& 0.4472& 0\\
0& 0& 1\end{array}\right).$$
Letting $Q_2:=\diag(I_3,V_{1})$ gives rise to Line 3.
$$\hat{B}:=Q_2^TBQ_2=\left(\begin{array}{ccc}
B_1&B_4&0\\
B_4^T&B_6&0\\
0^T&0&0
\end{array}\right),$$ where $B_4:=\left(\begin{array}{ccc}
-2.6833 & 1.3416\\
0&0\\
0&0
\end{array}\right) $ and $B_6:=\left(\begin{array}{ccc}
6&0\\
0&1
\end{array}\right) $. Note that since $B_5=0$,  $A$ and $B$ are SD and we skip Lines 4--6.
Following Line 7, letting $$Q_3:=\left(\begin{array}{ccc}
I_3&&\\
-B_6^{-1}B_4^T&I_2&\\
&&1
\end{array}\right)=\left(\begin{array}{cccccc}
1& 0& 0& 0& 0& 0\\
0& 1& 0& 0& 0& 0\\
0& 0& 1& 0& 0& 0\\
0.4472& 0& 0& 1& 0& 0\\
-1.3416& 0& 0& 0& 1& 0\\
0& 0& 0& 0& 0& 1
\end{array}\right)$$ further yields
\begin{eqnarray*}
&\tilde{B}:=&Q_3^T\hat{B}Q_3\\
&=&\left(\begin{array}{ccc}
B_1-B_4B_6^{-1}B_4^T&&B_5\\
&B_6&\\
B_5^T&&
\end{array}\right) \\
&=&\left(\begin{array}{cccccc}
-2& 2& 0& 0& 0& 0\\
2& 5& 1& 0& 0& 0\\
0& 1& 7& 0& 0& 0\\
0& 0& 0& 6& 0& 0\\
0& 0& 0& 0& 1& 0\\
0& 0& 0& 0& 0& 0
\end{array}\right).
\end{eqnarray*}
Then we goto Line 8. With  $$B_1-B_4B_6^{-1}B_4^T=\left(\begin{array}{ccc}
-2& 2& 0\\
2& 5& 1\\
0& 1& 7\\
\end{array}\right),~A_1=\left(\begin{array}{ccc}
1& 0& 0\\
0& 4& 0\\
0& 0& 9
\end{array}\right),$$ we obtain$$A_1^{-1}(B_1-B_4B_6^{-1}B_4^T)=\left(\begin{array}{ccc}
  -2  &  2   &      0\\
    0.5 &   1.25  &  0.25\\
         0 &   0.1111  &  0.7778
 \end{array}\right).$$
Applying Jordan decomposition gives rise to $J:=V_2^{-1}A_1^{-1}(B_1-B_4B_6^{-1}B_4^T)V_2$ = $\diag(1.5657,-2.2837,0.7458)$, where
$$V_{2}:=\left(\begin{array}{ccc}
    3.98 & 194.23  & -0.21\\
    7.09 & -27.55  & -0.29\\
    1 &   1  &  1
\end{array}\right).$$
Thus, $A_1$ and $B_1-B_4B_6^{-1}B_4^T$ are SD, and with $V_2^TA_1V_2=\diag(226,40772,9)$ and
$V_2^T(B_1-B_4B_6^{-1}B_4^T)V_2=\diag(354,-93111,7)$.
Next we go to Line 9. Denote $$Q_4:=\diag(V_{2},I_3)=\left(\begin{array}{cccccc}
3.98 & 194.23  & -0.21&&&\\
    7.09 & -27.55  & -0.29&&&\\
    1 &   1  &  1   &           &  &\\
 &  &  & 1    &              &      \\
& &   &      &       1       &       \\
    &      &  &          &                   & 1
\end{array}\right).$$
Note that in the last step, $V_2^TA_1V_2$ is already diagonal, so we do not need to apply spectral decomposition to $V_2^TA_1V_2$. Now by letting the congruent matrix $P:=Q_2Q_3Q_4$, we have $$P^TAP=\diag(226,40772,9,0,0,0),$$ and $$P^TBP=\diag(354,-93111,7,6,1,0).$$

\end{document}